\documentclass[12pt,reqno]{amsart}
\usepackage{amsmath, amsfonts, amssymb, amsthm, hyperref}
\usepackage{bm}
\def\l{\left}
\def\r{\right}
\def\bg{\bigg}
\def\({\bg(}
\def\){\bg)}

\def\f{\frac}

\def\eq{\equiv}

\def\1{{\bf 1}}

\theoremstyle{plain}
\newtheorem{theorem}{Theorem}[section]
\newtheorem{lemma}{Lemma}

\theoremstyle{definition}

\newtheorem*{Acks}{Acknowledgments}
\newtheorem*{Remark}{Remark}
\theoremstyle{Conjecture}

\def\<{\langle}
\def\>{\rangle}

\begin{document}
\title{On some conjectural congruences}

\begin{abstract}
In this paper, we confirm some congruences conjectured by V.J.W. Guo and M.J. Schlosser recently. For example, we show that for primes $p>3$,
$$
\sum_{k=0}^{p-1}(2pk-2k-1)\f{\l(\f{-1}{p-1}\r)_k^{2p-2}}{(k!)^{2p-2}}\eq0\pmod{p^5}.
$$
\end{abstract}
\author{Chen Wang}
\address{Department of Mathematics, Nanjing University, Nanjing 210093,
People's Republic of China}
\email{cwang@smail.nju.edu.cn}

\thanks{
\newline \indent This work was supported by the National Natural Science Foundation of China (grant no. 11571162).}
\maketitle

\section{Introduction}
\setcounter{lemma}{0} \setcounter{theorem}{0}
\setcounter{equation}{0}
In 2015, motivated by the well-known formula
$$
\lim_{n\rightarrow\infty}\l(1+\f{1}{n}\r)^n = e,
$$
Z.-W. Sun \cite{Su2} established some new congruences modulo prime powers. For example, he showed that for any prime $p>3$,
\begin{equation}\label{eq1}
\sum_{k=0}^{p-1}\binom{\f{-1}{p+1}}{k}^{p+1}\eq0\pmod{p^5}
\end{equation}
and
\begin{equation}\label{eq2}
\sum_{k=0}^{p-1}\binom{\f{1}{p-1}}{k}^{p-1}\eq\f{2}{3}p^4B_{p-3}\pmod{p^5},
\end{equation}
where $B_0,B_1,B_2,\ldots$ are the well-known Bernoulli numbers (cf. \cite{IR}).

Recently, V.J.W. Guo and M.J. Schlosser \cite{Guo} studied some interesting $q$-congruences for truncated basic hypergeometric series with the base being an even power of $q$. In their paper, as a corollary, they obtained that for any odd prime $p$,
$$
\sum_{k=0}^{p-1}k\cdot\f{\l(\f{1}{p+1}\r)_k^{p+1}}{(k!)^{p+1}}\eq0\pmod{p^3},
$$
where $(x)_k=x(x+1)\cdots(x+k-1)$ denotes the Pochhammer symbol. Note that
$$
\f{\l(\f{1}{p+1}\r)_k}{k!} = (-1)^k\binom{\f{-1}{p+1}}{k}.
$$
Thus this corollary is very similar to \eqref{eq1}. To refine this corollary, Guo and Schlosser conjectured that under the same condition,
\begin{equation}\label{conj1}
\sum_{k=0}^{p-1}k\cdot\f{\l(\f{1}{p+1}\r)_k^{p+1}}{(k!)^{p+1}}\eq\f{p^3}{4}\pmod{p^4}.
\end{equation}

We shall prove \eqref{conj1} by establishing the following extension.
\begin{theorem}\label{theorem1}
Let $p>5$ be a prime. Then
$$
\sum_{k=0}^{p-1}k\cdot\f{\l(\f{1}{p+1}\r)_k^{p+1}}{(k!)^{p+1}}\eq\f{p^3}{4}-\f{p^4}{8}+\f{3}{16}p^5-\f{1}{36}pB_{p-3}\pmod{p^6}.
$$
\end{theorem}
\begin{Remark}
We can directly verify that \eqref{conj1} holds for $p=5$.
\end{Remark}

Guo and Schlosser also posed some other conjectures which are similar to \eqref{conj1}. Here we shall confirm two of these conjectures.
\begin{theorem}\cite[Conjecture 2]{Guo}\label{theorem2}
Let $r$ be a positive integer and let $p$ be a prime with $p>2r+1$. Then
$$
\sum_{k=0}^{p-1}k^r\l(k+\f{1}{p+1}\r)^r\cdot\f{\l(\f{1}{p+1}\r)_k^{p+1}}{(k!)^{p+1}}\eq0\pmod{p^4}.
$$
\end{theorem}
\begin{theorem}\cite[Conjecture 4, (5.4)]{Guo}\label{theorem3}
Let $p>3$ be a prime. Then
$$
\sum_{k=0}^{p-1}(2pk-2k-1)\f{\l(\f{-1}{p-1}\r)_k^{2p-2}}{(k!)^{2p-2}}\eq0\pmod{p^5}.
$$
\end{theorem}
\begin{Remark}
Via a similar discussion as the proof of Theorem \ref{theorem3}, we can also prove that \cite[Conjecture 3, (5.2)]{Guo} holds modulo $p^5$. However, it is difficult to show this conjecture holds modulo $p^7$ by using the same method since the computation is very complicated.
\end{Remark}

In the next section, we shall prove Theorem \ref{theorem1} and Theorem \ref{theorem2}. The proof of Theorem \ref{theorem3} will be given in the last section.

\section{Proof of Theorems \ref{theorem1}--\ref{theorem2}}
\setcounter{lemma}{0} \setcounter{theorem}{0}
\setcounter{equation}{0}
To show theorems \ref{theorem1}--\ref{theorem2} we need the following lemmas.
\begin{lemma}\cite{Su2}\label{lemma1}
For primes $p>5$ we have
\begin{align*}
&\f{\l(\f{1}{p+1}\r)_k^{p+1}}{(k!)^{p+1}}-\prod_{j=1}^k\l(1-\f{p}{j}\r)\\
\eq&\f{p^5-p^4+p^3}{2}H_k^{(2)}+\f{p^4-3p^5}{6}H_k^{(3)}+\f{p^5}{12}H_k^{(4)}+\f{p(p^4-p^3)}{2}\sum_{1\leq i<j\leq k}\l(\f{1}{ij^2}+\f{1}{i^2j}\r)\\
&-\f{p^5}{6}\sum_{1\leq i<j\leq k}\l(\f{1}{i^3j}+\f{1}{ij^3}\r)+\f{p^5}{2}\sum_{1\leq i_1<i_2\leq k}\f{1}{i_1i_2}\l(\sum_{j=1}^k\f{1}{j^2}-\f{1}{i_1^2}-\f{1}{i_2^2}\r)\pmod{p^6},
\end{align*}
where $H_k^{(m)}=\sum_{i=1}^{k}1/i^m$ denotes the $k$-th harmonic number of order $m$.
\end{lemma}

\begin{lemma}\cite[pp. 125--126]{vW}\label{lemma2}
For any positive integer $n$ we have
$$
\sum_{k=0}^{n}\binom{n}{k}(-1)^kk^m=0\quad for\ all\ m=0,1,\ldots,n-1.
$$
\end{lemma}

\begin{lemma}\label{lemma3}
Let $p>3$ be a prime. Then
\begin{gather*}
\sum_{k=0}^{p-1}kH_k^{(2)}\eq-\f{p^2}{2}B_{p-3}-\f{p}{2}+\f{1}{2}\pmod{p^3},\quad\sum_{k=0}^{p-1}kH_k^{(3)}\eq\f{1}{3}pB_{p-3}\pmod{p^2},\\
\sum_{k=0}^{p-1}kH_k^{(4)}\eq0\pmod{p}.
\end{gather*}
\end{lemma}
\begin{proof}
For $m=2,3,4$, we have
\begin{align}\label{eqlem2.3}
\sum_{k=1}^{p-1}kH_k^{(m)}=&\sum_{k=1}^{p-1}k\sum_{j=1}^k\f{1}{j^m}=\sum_{j=1}^{p-1}\f{1}{j^m}\sum_{k=j}^{p-1}k=\sum_{j=1}^{p-1}\f{p^2-p-j^2+j}{2j^m}\notag\\
=&\f{p^2-p}{2}H_{p-1}^{(m)}-\f{1}{2}H_{p-1}^{(m-2)}+\f{1}{2}H_{p-1}^{(m-1)}.
\end{align}
In view of \cite[Lemma 2.1]{Su2}, we know that for primes $p>3$,
\begin{equation}\label{eqlem2.3'}
H_{p-1}\eq-\f{1}{3}p^2B_{p-3}\pmod{p^3},\quad H_{p-1}^{(2)}\eq\f{2}{3}pB_{p-3}\pmod{p^3},\quad H_{p-1}^{(3)}\eq0\pmod{p}.
\end{equation}
Combining \eqref{eqlem2.3'} with \eqref{eqlem2.3} we obtain Lemma \ref{lemma3}.
\end{proof}

\begin{lemma}\label{lemma4}
Let $p>3$ be a prime. Let
\begin{gather*}
\sigma_1 = \sum_{k=1}^{p-1}k\sum_{1\leq i<j\leq k}\l(\f{1}{ij^2}+\f{1}{i^2j}\r),\quad\sigma_2 = \sum_{k=1}^{p-1}k\sum_{1\leq i<j\leq k}\l(\f{1}{ij^3}+\f{1}{i^3j}\r),\\
\quad\sigma_3 = \sum_{k=1}^{p-1}k\sum_{1\leq i<j\leq k}\f{H_k^{(2)}}{ij}.
\end{gather*}
Then we have
\begin{gather*}
\sigma_1\eq-\f{1}{2}pB_{p-3}+\f{3}{4}(p-1)\pmod{p^2},\quad\sigma_2\eq\f{1}{2}B_{p-3}\pmod{p},\\
\sigma_3\eq\f{7}{8}+\f{1}{2}B_{p-3}\pmod{p}.
\end{gather*}
\end{lemma}
\begin{proof} In light of \eqref{eqlem2.3'} and noting that $\sum_{k=1}^{p-1}H_k/k^2\eq-\sum_{k=1}^{p-1}H_k^{(2)}/k\eq B_{p-3}\pmod{p}$ (cf. \cite[(5.4)]{ST}) we obtain that
\begin{align*}
\sigma_1 =& \sum_{1\leq i<j\leq p-1}\l(\f{1}{i^2j}+\f{1}{ij^2}\r)\f{p^2-p-j^2+j}{2}\\
\eq&-\f{p}{2}\sum_{1\leq i<j\leq p-1}\l(\f{1}{i^2j}+\f{1}{ij^2}\r)-\f{1}{2}\sum_{1\leq i<j\leq p-1}\l(\f{1}{i}+\f{j}{i^2}\r)+\f{1}{2}\sum_{1\leq i<j\leq p-1}\l(\f{1}{ij}+\f{1}{i^2}\r)\\
\eq&-\f{1}{2}pB_{p-3}-\f{1}{2}\sum_{i=1}^{p-1}\f{p-i-1}{i}-\f{1}{4}\sum_{i=1}^{p-1}\f{p^2-p-i^2-i}{i^2}+\f{H_{p-1}^2-H_{p-1}^{(2)}}{4}+\f{1}{2}\sum_{i=1}^{p-1}\f{p-i-1}{i^2}\\
\eq&-\f{1}{2}pB_{p-3}+\f{3}{4}(p-1)\pmod{p^2}.
\end{align*}
The proof of $\sigma_2\pmod{p}$ is similar. Below we consider $\sigma_3\pmod{p}$. By Lemma \ref{lemma3}, \cite[(2.7)]{Su2} and \cite[Lemma 2.1]{SZ} we arrive at
\begin{align*}
\sigma_3=&\sum_{1\leq i<j\leq p-1}\f{1}{ij}\sum_{k=j}^{p-1}kH_k^{(2)}\eq\sum_{1\leq i<j\leq p-1}\f{1}{ij}\l(\f{1}{2}-\sum_{k=1}^{j-1}kH_k^{(2)}\r)\\
\eq&-\sum_{1\leq i<j\leq p-1}\f{1}{ij}\l(\f{j^2-j}{2}H_{j-1}^{(2)}-\f{j-1}{2}+\f{1}{2}H_{j-1}\r)\\
=&-\sum_{1\leq i<j\leq p-1}\f{1}{ij}\l(\f{j^2}{2}H_j^{(2)}-\f{j}{2}H_j^{(2)}-\f{j}{2}+\f{1}{2}H_j\r)\\
\eq&\sum_{i=1}^{p-1}\f{i+1}{4}H_i^{(2)}-\f{p-1}{4}+\f{3}{4}\sum_{k=1}^{p-1}\f{H_k}{k}-\f{1}{2}\sum_{i=1}^{p-1}\f{H_{i}^{(2)}}{i}+\f{1}{2}+\f{1}{4}\sum_{i=1}^{p-1}\f{H_i^2+H_i^{(2)}}{i}\\
\eq&\f{7}{8}+\f{1}{2}B_{p-3}\pmod{p}.
\end{align*}

Now the proof of Lemma \ref{lemma4} is complete.
\end{proof}

\medskip
\noindent{\it Proof of Theorem \ref{theorem1}}. In view of Lemmas \ref{lemma1}--\ref{lemma4}, we have
\begin{align*}
&\sum_{k=0}^{p-1}k\f{\l(\f{1}{p+1}\r)_k^{p+1}}{(k!)^{p+1}}=\sum_{k=0}^{p-1}k\f{\l(\f{1}{p+1}\r)_k^{p+1}}{(k!)^{p+1}}-\sum_{k=1}^{p-1}(-1)^kk\binom{p-1}{k}\\
\eq&\f{p^5-p^4+p^3}{2}\l(-\f{p^2}{2}B_{p-3}-\f{p}{2}+\f{1}{2}\r)+\f{p^4-3p^5}{6}\cdot\f{1}{3}pB_{p-3}\\
&+\f{p(p^4-p^3)}{2}\l(-\f{1}{2}pB_{p-3}+\f{3}{4}(p-1)\r)-\f{1}{12}p^5B_{p-3}+\f{7}{16}p^5\\
\eq&\f{p^3}{4}-\f{p^4}{8}+\f{3}{16}p^5-\f{1}{36}pB_{p-3}\pmod{p^6}.
\end{align*}
This proves Theorem \ref{theorem1}.\qed

\medskip
\noindent{\it Proof of Theorem \ref{theorem2}}. The case $p=5$ can be verified directly. Now we suppose that $p>5$. Since $p>2r+1$, it is easy to see that
$$
\sum_{k=0}^{p-1}k^r\l(\f{1}{p+1}+k\r)^r(-1)^k\binom{p-1}{k}=0
$$
with the help of Lemma \ref{lemma2}. Thus by Lemma \ref{lemma1} we arrive at
\begin{align*}
&\sum_{k=0}^{p-1}k^r\l(k+\f{1}{p+1}\r)^r\cdot\f{\l(\f{1}{p+1}\r)_k^{p+1}}{(k!)^{p+1}}\\
=&\sum_{k=0}^{p-1}k^r\l(k+\f{1}{p+1}\r)^r\cdot\l(\f{\l(\f{1}{p+1}\r)_k^{p+1}}{(k!)^{p+1}}-(-1)^k\binom{p-1}{k}\r)\\
\eq&\f{p^3}{2}\sum_{k=0}^{p-1}k^r(k+1)^rH_k^{(2)}\pmod{p^4}.
\end{align*}
Noting that
$$
H_{p-1-k}^{(2)}\eq-H_k^{(2)}\pmod{p},
$$
we have
\begin{align*}
&\sum_{k=0}^{p-1}k^r(k+1)^rH_k^{(2)}=\sum_{k=0}^{p-1}(p-1-k)^r(p-k)^rH_{p-1-k}^{(2)}\\
\eq&-\sum_{k=0}^{p-1}k^r(k+1)^rH_k^{(2)}\eq0\pmod{p}.
\end{align*}
This concludes the proof of Theorem \ref{theorem2}.\qed

\section{Proof of Theorem \ref{theorem3}}.
We need the following preliminary results.
\setcounter{lemma}{0} \setcounter{theorem}{0}\setcounter{equation}{0}
\begin{lemma}\label{lemma3_1}For any nonnegative integer $n$ we have the following identities.
\begin{gather}
\label{id1}\sum_{k=0}^n(-1)^k\binom{2n+1}{k}=\f{(-1)^n(n+1)}{2n+1}\binom{2n+1}{n},\\
\label{id2}\sum_{k=0}^n(-1)^kk\binom{2n+1}{k}=\f{(-1)^n(n+1)}{2}\binom{2n+1}{n}.
\end{gather}
\end{lemma}
\begin{proof}
Denote the left-hand side of \eqref{id1} and \eqref{id2} by $S_1(n)$ and $S_2(n)$ respectively. Via Zeilberger's algorithm \cite{PWZ}, we find that $S_1(n)$ satisfies
$$
(-1-2n)S_1(n)+(3+2n)S_1(1+n) = -\f{(-1)^n(3+n)(7+5n)}{4(3+2n)}\binom{4+2n}{1+n}
$$
and $S_2(n)$ satisfies
$$
-2nS_2(n)+(2+2n)S_2(1+n)=-\f{(-1)^n(2+n)(3+n)(6+5n)}{2(3+2n)}\binom{3+2n}{n}.
$$
Then we can easily show these two identities by induction on $n$.
\end{proof}

\begin{lemma}\label{lemma3_2}
Let $p>3$ be a prime. Then
$$
\sum_{k=0}^{p-1}k\sum_{1\leq i<j\leq k}\f{1}{i^2j^2}\eq-\f{1}{2}B_{p-3}\pmod{p}.
$$
\end{lemma}
\begin{proof} Noting \eqref{eqlem2.3'} and \cite[(5.4)]{ST} we have
\begin{align*}
&\sum_{k=0}^{p-1}k\sum_{1\leq i<j\leq k}\f{1}{i^2j^2}=\sum_{1\leq i<j\leq p-1}\f{1}{i^2j^2}\sum_{k=j}^{p-1}k=\sum_{1\leq i<j\leq p-1}\f{1}{i^2j^2}\f{p^2-p-j^2+j}{2}\\
\eq&\sum_{1\leq i<j\leq p-1}\f{-j^2+j}{2i^2j^2}\eq\sum_{i=1}^{p-1}\f{1}{i^2}\sum_{j=i+1}^{p-1}\f{1}{j}\eq-\sum_{i=1}^{p-1}\f{H_i}{2i^2}\eq-\f{1}{2}B_{p-3}\pmod{p}.
\end{align*}
\end{proof}

\medskip
\noindent{\it Proof of Theorem \ref{theorem3}}. One can directly check that Theorem \ref{theorem3} holds for $p=5$. Below we assume that $p>5$. For each $k=\{1,\ldots,p-1\}$, we have
\begin{align*}
&\f{\l(\f{-1}{p-1}\r)_k^{2p-2}}{(k!)^{2p-2}}=\prod_{j=1}^k\l(\f{\f{1}{p-1}-j+1}{j}\r)^{2p-2}=\prod_{j=1}^k\l(1-\f{p}{(p-1)j}\r)^{2p-2}\\
\eq&\prod_{j=1}^k\l(1-\f{2p}{j}+\binom{2p-2}{2}\f{p^2}{(p-1)^2j^2}-\binom{2p-2}{3}\f{p^3}{(p-1)^3j^3}+\binom{2p-2}{4}\f{p^4}{(p-1)^4j^4}\r)\\
\eq&\prod_{j=1}^k\l(1-\f{2p}{j}+\f{p^4+p^3+3p^2}{j}-\f{10p^4+12p^3}{3j^3}+\f{5p^4}{j^4}\r)\pmod{p^5}
\end{align*}
and so
\begin{align*}
&\f{\l(\f{-1}{p-1}\r)_k^{2p-2}}{(k!)^{2p-2}}-\prod_{j=1}^k\l(1-\f{2p}{j}\r)\\
\eq&(p^4+p^3+3p^2)H_k^{(2)}-\f{10p^4+12p^3}{3}H_k^{(3)}+5p^4H_k^{(4)}\\
&+12p^4\sum_{1\leq i<j\leq k}\f{1}{ij}\l(H_k^{(2)}-\f{1}{i^2}-\f{1}{j^2}\r)+9p^4\sum_{1\leq i<j\leq k}\f{1}{i^2j^2}\\
&-(2p^4+6p^3)\sum_{1\leq i<j\leq k}\l(\f{1}{ij^2}+\f{1}{i^2j}\r)+8p^4\sum_{1\leq i<j\leq k}\l(\f{1}{ij^3}+\f{1}{i^3j}\r)\pmod{p^5}.
\end{align*}
Thus in view of Lemma \ref{lemma3}--\ref{lemma4}, Lemma \ref{lemma3_2} and \cite[Lemmas 2.1\&2.2]{Su2}, we obtain that
\begin{align}\label{keystep1}
\sum_{k=0}^{p-1}(2pk-2k-1)\l(\f{\l(\f{-1}{p-1}\r)_k^{2p-2}}{(k!)^{2p-2}}-(-1)^k\binom{2p-1}{k}\r)\eq-8p^4-4p^3-3p^2\pmod{p^5}.
\end{align}
In 1990, Glaisher \cite{G1,G2} showed that
$$
\binom{2p-1}{p-1}\eq1-\f{2}{3}p^3B_{p-3}\pmod{p^4}.
$$
This together with Lemma \ref{lemma3_1} gives that
\begin{align}\label{keystep2}
&\sum_{k=0}^{p-1}(2pk-2k-1)(-1)^k\binom{2p-1}{k}\notag\\
=&p(p-1)\binom{2p-1}{p-1}-\f{p}{2p-1}\binom{2p-1}{p-1}\notag\\
\eq&8p^4+4p^3+3p^2\pmod{p^5}.
\end{align}
Combining \eqref{keystep1} and \eqref{keystep2} we immediately obtain Theorem \ref{theorem3}.\qed

\begin{Acks}
The author would like to thank Dr. Guo-Shuai Mao for his helpful comments.
\end{Acks}

\end{document}